\documentclass[12pt]{amsart}
\usepackage[utf8]{inputenc}
\usepackage[ngerman,english]{babel}
\usepackage{fullpage}
\usepackage{graphicx}
\usepackage{amsmath}             
\usepackage{amsfonts} 
\usepackage{tikz-cd} 
\usepackage[nomain,style=long3col,sort=use,section=section]{glossaries}

\usepackage{setspace}

\usepackage{amsthm}               
\usepackage{amssymb}              
\usepackage{tikz}                       
\usepackage[all]{xy}
\usetikzlibrary{plotmarks}
\usetikzlibrary{shapes}
\usetikzlibrary{arrows}

\usepackage{graphicx}
\usepackage{caption}
\usepackage{subcaption}
\usepackage{upgreek}
\usepackage{ gensymb }
\usepackage{capt-of}
\usepackage{ stmaryrd }
\usepackage[pdftex]{hyperref}
\hypersetup{
    colorlinks=true,
    linkcolor={green!30!black},
    citecolor={blue!50!black},
    urlcolor={blue!80!black}
    allbordercolors={1 0 0}
}
\usepackage{xcolor}
\usepackage{color}

\numberwithin{equation}{section}
\theoremstyle{plain} 

\newtheorem{theorem}{Theorem}[section]

\newtheorem{corollary}[theorem]{Corollary}
\newtheorem{lemma}[theorem]{Lemma}
\newtheorem{proposition}[theorem]{Proposition}
\newtheorem{problem}[theorem]{Problem}

\theoremstyle{definition}

\newtheorem{remark}[theorem]{Remark}
\newtheorem{example}[theorem]{Example}

\newtheorem{definition}[theorem]{Definition}

\DeclareMathOperator{\conv}{conv}
\DeclareMathOperator{\cone}{cone}
\DeclareMathOperator{\Spec}{Spec}

\title{Rigid Gorenstein toric Fano varieties arising from directed graphs}
\author{Selvi Kara}
\address{Department of Mathematics, University of Utah, 155 1400 E, Salt Lake City, UT 84112}
\email{selvi@math.utah.edu}

\author{Irem Portakal}
\address{Max-Planck-Institut f\"{u}r Mathematik in den Naturwissenschaften, Inselstraße 22, 04103, Leipzig, Germany.}
\email{mail@irem-portakal.de} 

\author{Akiyoshi Tsuchiya}
\address{Graduate school of Mathematical Sciences,
University of Tokyo,
Komaba, Meguro-ku, Tokyo 153-8914, Japan}
\email{akiyoshi@ms.u-tokyo.ac.jp}

\date{\today}


\begin{document}
\begin{abstract}
A directed edge polytope $\mathcal{A}_G$ is  a lattice polytope arising from root system $A_n$ and a finite directed graph $G$. 
If every directed edge of $G$ belongs to a directed cycle in $G$, then $\mathcal{A}_G$ is terminal and reflexive, that is, one can associate this polytope to a Gorenstein toric Fano variety $X_G$ with terminal singularities.
It is shown by Totaro that a toric Fano variety which is smooth in codimension $2$ and $\mathbb{Q}$-factorial in codimension $3$ is rigid.
In the present paper, we classify all directed graphs $G$ such that $X_G$ is a toric Fano variety which is smooth in codimension $2$ and $\mathbb{Q}$-factorial in codimension $3$.
\end{abstract}

\maketitle

\section{Introduction}
\subsection{Directed edge polytopes and symmetric edge polytopes}
A \textit{lattice} polytope is a convex polytope all of whose vertices have integer coordinates.
Let $G=(V(G),A(G))$ be a finite directed graph on the vertex set $V(G)=[n]:=\{1,\ldots, n\}$ with the directed edge set $A(G)$.
For a directed edge $e = (i,j)$ of $G$, we define $\rho (e) \in \mathbb{R}^n$ by setting $\rho (e) = \textbf{e}_i -\textbf{e}_j$. The \textit{directed edge polytope} of $G$, denoted by $\mathcal{A}_G$, is the lattice polytope defined as $$ \mathcal{A}_G  = \conv \{ \rho (e) : e \in A(G)\} \subset \mathbb{R}^n.$$
A related polytope is the \textit{symmetric edge polytope} of $G$, denoted by $\mathcal{P}_G $,  and defined as
$$ \mathcal{P}_G  = \conv \{ \pm \rho (e) : e \in A(G)\} \subset \mathbb{R}^n.$$
Note that one can define the symmetric edge polytope based on the underlying simple graph of $G$, denoted by $G^{\rm un}$, since both directions for each edge contribute to the vertices of $ \mathcal{P}_G$.
Here, $G^{\rm un}$ has an edge $\{i,j\}$ if and only if $(i,j)$ or $(j,i)$ is a directed edge of $G$.
Directed edge polytopes are first introduced in \cite{ohsugi2002hamiltonian} by Ohsugi and Hibi for a tournament graph $G$. 
A directed graph $G$ is called \emph{symmetric} if both  $(i,j)$ and $(j,i) \in A(G)$ whenever $\{i,j\}$ is an edge of $G^{\text{un}}$. Notions of directed edge polytopes and symmetric edge polytopes coincide when $G$ is  a symmetric directed graph, i.e., $ \mathcal{A}_G = \mathcal{P}_G$. Symmetric edge polytopes are of interest in   many fields such as commutative algebra (\cite{OHcentrally}), algebraic geometry (\cite{dimpoly}) algebraic combinatorics (\cite{hijomi,OT1,OT2}) and number theory (\cite{Riemannhyp,matsui2011roots,zero}).
One of the reasons the symmetric edge polytopes gained much attraction is due to its connection to the Kuramoto model (\cite{Kuramoto}), which describes the behavior of interacting oscillators (\cite{CRM}).\\

If one assumes that every directed edge of $G$ belongs to a directed cycle in $G$, then the origin belongs to the relative interior of $\mathcal{A}_G$. 
In particular, $\mathcal{A}_G$ is a terminal reflexive polytope \cite[Proposition 1.4]{dimpoly}.
We denote $X_G$ as the projective toric variety associated to the spanning (or face) fan of $\mathcal{A}_G \subseteq N_{\mathbb{R}} \cong \mathbb{R}^n$. Then $X_G$ is a Gorenstein toric Fano variety with terminal singularities. 
Gorenstein toric Fano varieties are of interest in algebraic geometry and mirror symmetry (\cite{mirror,GoresteinFano}). We assume that the reader is familiar with toric varieties. For an introduction of the toric varieties and the notations, we refer the reader to \cite{toricvarieties}. In particular the letter $N \cong \mathbb{Z}^n$ stands for a lattice and $M=\text{Hom}_{\mathbb{Z}}(N, \mathbb{Z})$ is its dual lattice. We denote their associated vector spaces as $N_{\mathbb{R}}:=N \otimes_{\mathbb{Z}} \mathbb{R}$ and $M_{\mathbb{R}}:=M \otimes_{\mathbb{Z}} \mathbb{R}$.

\subsection{Deformation theory}
Let $X$ be a scheme of finite type over $\mathbb{C}$ and let $A$ be an Artinian algebra over $\mathbb{C}$. An infinitesimal deformation of $X$
over $A$ is defined as the following cartesian diagram: 
\[\begin{tikzcd}
X \arrow{r}{} \arrow[swap]{d}{} & \mathcal{X} \arrow{d}{\pi} \\
0 \arrow{r}{} & \Spec(A)
\end{tikzcd}
\]
where $\pi$ is flat. Let $\pi':  \mathcal{X}' \rightarrow \Spec(A)$ be another deformation of $X$ over $\Spec(A)$. We say that $\pi$ and $\pi'$ are isomorphic, if there exists a map $\mathcal{X} \rightarrow \mathcal{X}'$ over $\Spec(A)$ inducing the identity on $X$. Let $\text{Def}_X$ be a functor such that $\text{Def}_X(A)$ is the set of deformations of $X$ over $\Spec(A)$ modulo isomorphisms. $X$ is called \emph{rigid}, if the first-order deformation space $T_X^1 : = \text{Def}_X(\mathbb{C}[t]/t^2) = 0$. This implies that $X$ has no nontrivial infinitesimal deformations. It turns out that the toric varieties arising from graphs recover interesting rigid examples. One of them is the famous rigid singularity which is the cone over the Segre embedding $\mathbb{P}^m \times \mathbb{P}^n$ in $\mathbb{P}^{(m+1)(n+1)-1}$ except for $m=n=1$. Equivalently, this is the affine toric variety arising from the complete bipartite graph $K_{m+1,n+1}$. The rigidity of toric varieties associated to undirected graphs has been first studied \cite{bigdeliherzoglu} for complete bipartite graphs with one edge removals. A generalization of this family and a sufficient condition for rigidity have been examined in \cite{portakal}. In the present paper, we will discuss infinitesimal deformations of the Gorenstein toric Fano varieties associated to directed edge polytopes. This connection between three different areas of mathematics gives us the opportunity to present high dimensional concrete examples of rigid Gorenstein toric Fano varieties purely in terms of graphs, which is in general a hard computational problem (see Section \ref{section5}).

\subsection{Smooth toric Fano varieties arising from directed graphs}
For toric Fano varieties, a sufficient condition for their rigidity is known. In fact, $\mathbb{Q}$-factorial toric Fano varieties with terminal singularities, in particular, smooth toric Fano varieties are rigid {\cite[Theorem 1.4]{FernexHacon}, \cite[Theorem 3.2]{bienbrion}}.
In \cite{dimpoly}, Higashitani classified directed graphs $G$ such that $X_G$ is $\mathbb{Q}$-factorial. All undefined terms are
specified in the sections below.
\begin{theorem}[{\cite[Theorem 2.2]{dimpoly}}]
	Let $G$ be a finite directed graph such that every directed edge belongs to a directed cycle in $G$. Then the following arguments are equivalent:
		\begin{enumerate}
		\item $X_G$ is smooth;
		\item $X_G$ is $\mathbb{Q}$-factorial;
		\item $G$ possesses no homogeneous cycle $C=(e_1,\ldots,e_l)$ such that 
		\[
		\mu_C(i_1)-\mu_C(i_b) \leq {\rm dist}_G(i_a,i_b)
		\]
		for all $1 \leq a,b \leq l$, where $e_j$ is $(i_j, i_{j+1})$ or $(i_{j+1},i_j)$ for $1 \leq j \leq l$ with $i_{l+1}=i_i$.
	\end{enumerate}
In particular, $X_G$ is rigid.
\end{theorem}

For symmetric directed graphs, i.e., for  symmetric edge polytopes, one obtains the following.
\begin{corollary}[{\cite[Corollary 2.2]{dimpoly}}]
Let $G$ be a finite symmetric directed graph.
Then the following arguments are equivalent:
\begin{enumerate}
    \item $X_G$ is smooth;
    \item $X_G$ is $\mathbb{Q}$-factorial;
    \item $G^{\rm un}$ has no even cycle as subgraphs.
\end{enumerate}
In particular, $X_G$ is rigid.
\end{corollary}

\subsection{Rigid Gorenstein toric Fano varieties arising from directed graphs}
The most general rigidity theorem for toric Fano varieties known to this date is the following result of Totaro: 
\begin{theorem}{\cite[Theorem 5.1]{totaro}}\label{totaro}
A toric Fano variety which is smooth in codimension 2 and $\mathbb{Q}$-factorial in codimension 3 is rigid.
\end{theorem}
Combinatorially one can interpret this result in terms of the associated Fano polytope. The toric Fano variety $X_P$ is smooth in codimension 2 if and only if the edges of $P$ has lattice length 1 and they are contained in a hyperplane of height 1 with respect to the origin. The $\mathbb{Q}$-factorial in codimension 3 means that the two dimensional faces of $P$ are all triangles. \\

In the present paper, we classify all directed graphs $G$ such that $X_G$ is a Gorenstein toric Fano variety which is smooth in codimension 2 and $\mathbb{Q}$-factorial in codimension 3.  Note that since $\mathcal{A}_G$ is reflexive and terminal, $X_G$ is smooth in codimension $2$.  The following is the main theorem of the present paper.
\begin{theorem}
\label{thm:main}
Let $G$ be a finite directed graph such that every directed edge of $G$ belongs to a directed cycle in $G$.
Then the following arguments are equivalent:
\begin{enumerate}
	%	\item $X_G$ is rigid;
		\item $X_G$ is smooth in codimension $2$ and $\mathbb{Q}$-factorial in codimension $3$;
		\item $G$ satisfies both of the following:
	\begin{itemize}
    \item $G$ has no directed subgraph $C_1$ whose directed edge set is 
\[
\{(i_1,i_2), (i_1,i_4), (i_3,i_2), (i_3,i_4)\};
\]
\item For any directed subgraph $C_2$ of $G$ whose directed edge set is
\[
\{(i_1,i_2), (i_2,i_3), (i_1,i_4), (i_4,i_3)\},
\]
it follows that $(i_1,i_3)$ is a directed edge of $G$ or there exists a vertex $j \notin \{i_2,i_4\}$ in $G$ such that $(i_1,j)$ and $(j,i_3)$ are directed edges of $G$.
\end{itemize}
\end{enumerate}
In this case, $X_G$ is rigid.
\end{theorem}

For symmetric directed graphs, i.e., for  symmetric edge polytopes, we obtain the following.

\begin{corollary}\label{corollaryforsymmetric}
Let $G$ be a finite symmetric directed graph.
Then the following arguments are equivalent:
\begin{enumerate}
	%	\item $X_G$ is rigid;
		\item $X_G$ is smooth in codimension $2$ and $\mathbb{Q}$-factorial in codimension $3$;
		\item $G^{\rm un}$ has no $4$-cycle as a subgraph.
\end{enumerate}
In this case, $X_G$ is rigid.
\end{corollary}
The paper is organized as follows:
In Section 2, we recall a connection between lattice polytopes and toric Fano varieties.
In Section 3, we present certain characterization of faces of directed edge polytopes.
The proof of Theorem~\ref{thm:main} will be given in Section 4.
Finally, we give some examples of Gorenstein toric Fano varieties which are not $\mathbb{Q}$-factorial but rigid, and concluding remarks.

\section{Fano polytopes}
In this section, we recall a connection between lattice polytopes and toric Fano varieties.
Let $\mathcal{P} \subset \mathbb{R}^n$ be a full-dimensional lattice polytope. 
\begin{itemize}
    \item We say that $\mathcal{P}$ is a \emph{Fano} if the origin of $\mathbb{R}^n$ belongs to the interior of $\mathcal{P}$ and the vertices of $\mathcal{P}$ are primitive lattice points in $\mathbb{Z}^n$.
    \item  A Fano polytope $\mathcal{P}$ is called \emph{terminal} if every lattice point on the boundary is a vertex. 
    \item A Fano polytope $\mathcal{P}$ is said to be \emph{reflexive} if each facet of $\mathcal{P}$ has lattice distance one from the origin. Equivalently,
      its dual polytope 
    $$\mathcal{P}^{\vee} =\{x\in \mathbb{R}^n: \langle \mathbf{x}, \mathbf{y} \rangle \leq 1 \text{ for all } y \in \mathcal{P} \}$$
    is a lattice polytope.
    \item A Fano polytope  is called \emph{$\mathbb{Q}$-factorial} if it is simplicial.
    \item A Fano polytope is called \emph{smooth} if the vertices of each facet form a $\mathbb{Z}$-basis of $\mathbb{Z}^n$. 
\end{itemize}
We recall algebro-geometric interpretations of these polytopes.
For a Fano polytope $\mathcal{P}$, denote $X_\mathcal{P}$ the normal toric variety associated to the spanning fan of $\mathcal{P}$.
Then $X_{\mathcal{P}}$ is a toric Fano variety. Conversely, every toric Fano variety arises in this way from a Fano polytope (see \cite[\S 8.3]{toricvarieties}).
\begin{itemize}
	\item A Fano polytope $\mathcal{P}$ is terminal if and only if $X_{\mathcal{P}}$ has at worst terminal singularities.
	\item A Fano polytope $\mathcal{P}$ is reflexive if and only if $X_{\mathcal{P}}$ is Gorenstein.
	\item A Fano polytope $\mathcal{P}$ is $\mathbb{Q}$-factorial if and only if $X_{\mathcal{P}}$ is $\mathbb{Q}$-factorial.
	\item A Fano polytope $\mathcal{P}$ is smooth if and only if $X_{\mathcal{P}}$ is smooth.
\end{itemize}

Two lattice polytopes  $\mathcal{P} \subset \mathbb{R}^n$ and $\mathcal{Q} \subset \mathbb{R}^{m}$ are said to be \textit{unimodularly equivalent}, denoted by $\mathcal{P} \cong \mathcal{Q}$,  if there exists an affine map from the affine span ${\rm aff}(\mathcal{P})$ of $\mathcal{P}$ to the affine span ${\rm aff}(\mathcal{Q})$ of $\mathcal{Q}$ that maps $\mathbb{Z}^n \cap {\rm aff}(\mathcal{P})$ bijectively onto $\mathbb{Z}^m \cap {\rm aff}(\mathcal{Q})$ and maps $\mathcal{P}$ to $\mathcal{Q}$.
Each lattice polytope is unimodularly equivalent to a full-dimensional lattice polytope.
We say that a lattice polytope is \textit{Fano, terminal, reflexive, $\mathbb{Q}$-factorial, smooth} if it contains the origin in the interior and it is unimodularly equivalent to a full-dimensional Fano, terminal, reflexive, $\mathbb{Q}$-factorial, smooth polytope, respectively. The authors of \cite{matsui2011roots} classified all graphs whose directed edge polytopes are Fano.

\begin{proposition}[{\cite[Proposition 3.2]{matsui2011roots}}]\label{terminalreflexive}
Let $G$ be a finite directed graph graph. Then the following arguments are equivalent:
\begin{enumerate}
	\item $\mathcal{A}_G$ is Fano; 
	\item $\mathcal{A}_G$ is terminal reflexive;
	\item Every directed edge of $G$ belongs to a directed cycle in $G$.
\end{enumerate} 
In this case, 
$X_{G}$ is a Gorenstein toric Fano variety with terminal singularities.
\end{proposition}

Given two Fano polytopes $\mathcal{P} \subset \mathbb{R}^n$ and $\mathcal{Q} \subset \mathbb{R}^m$, set 
\[
\mathcal{P} \oplus \mathcal{Q}:={\rm conv}\{ (\mathcal{P}, \mathbf{0}) \cup (\mathbf{0}, \mathcal{Q})\} \subset \mathbb{R}^{n+m}.
\]
We call $\mathcal{P} \oplus \mathcal{Q}$ the \textit{free sum} of $\mathcal{P}$ and $\mathcal{Q}$.
Then one has $X_{\mathcal{P} \oplus \mathcal{Q}} \cong X_{\mathcal{P}} \times X_{\mathcal{Q}}$.
In particular, $\mathcal{P} \oplus \mathcal{Q}$ is smooth if and only if both $\mathcal{P}$ and $\mathcal{Q}$ are smooth. Thus, if $G$ is not connected, we have the following result by definition.
\begin{proposition}\label{disconnectedcase}
Let $G$ be a finite directed graph with connected components $G_1,\ldots,G_r$ such that $\mathcal{A}_G$ is terminal and reflexive. Then one has
\[
\mathcal{A}_G \cong \mathcal{A}_{G_1} \oplus \cdots \oplus \mathcal{A}_{G_r}.
\]
In particular, we obtain
\[
X_{G} \cong X_{{G_1}} \times \cdots \times X_{{G_r}}.
\]
\end{proposition}

We have the following combinatorial criterion for rigidity of terminal and reflexive polytopes.
\begin{proposition}\label{triangle2facerigid}
	Let $\mathcal{P}$ be a terminal reflexive polytope.
	If all $2$-faces of $\mathcal{P}$ are triangles, then $X_{\mathcal{P}}$ is rigid. 
\end{proposition}
\begin{proof}
By Theorem \ref{totaro}, a toric Fano variety which is smooth in codimension $2$ and $\mathbb{Q}$-factorial in codimension $3$ is rigid. Since $\mathcal{P}$ is terminal and reflexive, each edge of $\mathcal{P}$ has lattice length $1$ and is contained in some hyperplane which has height $1$ with respect to the origin. It follows that for rigidity it is enough to have triangle 2-faces, which is equivalent to say that $X_{\mathcal{P}}$ is smooth in codimension 3.
\end{proof}
The reason that we only consider connected graphs for $2$-faces is explained with the following result.
\begin{proposition}
Let $\mathcal{P}$ and $\mathcal{Q}$ be terminal reflexive polytopes. If all $2$-faces of $\mathcal{P}$ and $\mathcal{Q}$ are triangles, 
all $2$-faces of $\mathcal{P}\oplus\mathcal{Q}$ are also triangles. 
In this case,  $X_{\mathcal{P}\oplus\mathcal{Q}}$ is rigid.
\end{proposition}
\begin{proof}
 This follows from the fact that faces of $\mathcal{P} \oplus \mathcal{Q}$ are convex hulls of faces of $\mathcal{P}$ and $\mathcal{Q}$ and the fact that $\mathcal{P}\oplus\mathcal{Q}$ is a terminal reflexive polytope. 
\end{proof}

\section{Faces of directed edge polytopes}
In this section, we discuss faces of terminal reflexive directed edge polytopes.
Let $G$ be a connected finite directed graph on the vertex set $[n]$ and the directed edge set $A(G)$. We assume that every directed edge of $G$ belongs to a directed cycle in $G$, i.e.\ $\mathcal{A}_G$ is terminal and reflexive by Proposition \ref{terminalreflexive}.
First, we show that 
each proper face of $\mathcal{A}_G$ is the directed edge polytope of a finite acyclic directed graph.
Here a finite directed graph is called \textit{acyclic} if it has no directed cycles. 

\begin{lemma}
\label{lem:symedgeface}
Let $G$ be a connected finite directed graph such that $\mathcal{A}_G$ is terminal and reflexive. Then each proper face of the directed edge polytope $\mathcal{A}_G$ is the directed edge polytope of a finite acyclic directed subgraph $H$ of $G$.
\end{lemma}

\begin{proof}
	It is clear that each proper face of $\mathcal{A}_G$ is the directed edge polytope $\mathcal{A}_H$ of a directed subgraph $H$ of $G$. 
	If $H$ is not acyclic, then there exists a directed cycle
	\[
	(i_1,i_2),(i_2,i_3),\ldots,(i_r,i_1)
	\]
	of $H$. 
	Hence $\mathbf{e}_{i_1}-\mathbf{e}_{i_2},\mathbf{e}_{i_2}-\mathbf{e}_{i_3}, \ldots, \mathbf{e}_{i_r}-\mathbf{e}_{i_1}$ belong to $\mathcal{A}_H$.
	Then since 
	\[
	\dfrac{1}{r}((\mathbf{e}_{i_1}-\mathbf{e}_{i_2})+(\mathbf{e}_{i_2}-\mathbf{e}_{i_3})+ \cdots+(\mathbf{e}_{i_r}-\mathbf{e}_{i_1}))=\mathbf{0},
	\]
	the origin also belongs to $\mathcal{A}_H$.
	However, this contradicts that the origin belongs to the relative interior of $\mathcal{A}_G$ and $\mathcal{A}_H$ is a proper face of $\mathcal{A}_G$.
	Therefore, $H$ is acyclic.
\end{proof}
In what follows, we see that a cycle of $G$ determines a proper face of $\mathcal{A}_G$. We first introduce some terminology and notation to state this fact. For a directed edge $e=(i,j)$ of $G$, denote $e^{\rm un}$ the undirected edge $\{i,j\}$ of $G^{\rm un}$. 
A sequence $\Gamma=\{e_1,\ldots, e_l\}$ of directed edges of $G$ is called a \emph{cycle} if $\{e^{\rm un}_1,\ldots,e^{\rm un}_l\}$ forms a cycle in $G^{\rm un}$. 
Let $\Gamma=(e_1, \ldots, e_l)$ be a cycle in $G$ such that $e^{\rm un}_j=\{i_j,i_{j+1}\}$ for each $j$, where $i_{l+1}=i_1$.
Then we define
\begin{eqnarray*}
\Delta_{\Gamma}^+ &= \{e_j \in \{e_1,  \ldots, e_l\} :  e_j = (i_j, i_{j+1})\},\\
\Delta_{\Gamma}^- &= \{e_j \in \{e_1,  \ldots, e_l\} :  e_j  = (i_{j+1}, i_j)\}.
\end{eqnarray*}
 A cycle $\Gamma$ is called \emph{homogeneous} if $|\Delta_{\Gamma}^+| = |\Delta_{\Gamma}^-|$  and \emph{nonhomogeneous} if $|\Delta_{\Gamma}^+| \neq |\Delta_{\Gamma}^-|$. We note that the two directed edges $(i,j)$ and $(j,i)$ form a non-homogeneous cycle of length two, although they do not form a cycle in $G^{\rm un}$. 
 \begin{example}
Let $G$ be the directed graph given in  Figure \ref{fig:f0}. Consider the  cycles $\Gamma_1=\{i_1,i_2,i_5,i_6\}$ and $\Gamma_2=\{i_2,i_3,i_4,i_5\}$. Then $\Gamma_1$ is a homogeneous cycle whereas $\Gamma_2$ is not because $|\Delta_{\Gamma_1}^+| = 2= |\Delta_{\Gamma_1}^-|$ and $|\Delta_{\Gamma_2}^+| =3 \neq 1= |\Delta_{\Gamma_1}^-|.$
\begin{figure}[h]
    \centering
    \includegraphics[width=0.4\textwidth]{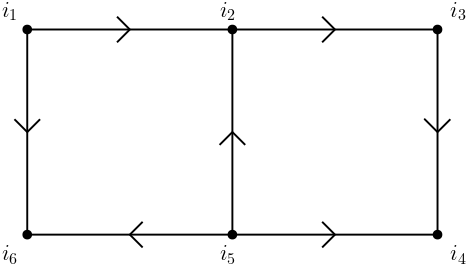}
    \caption{}
    \label{fig:f0}
\end{figure}
\end{example}

Assume that $G$ has a homogeneous cycle $\Gamma$ on the vertices $\{i_1,\ldots,i_l\}$.
Then there exists a unique function
\[
\mu_{\Gamma}: \{i_1,\ldots,i_l\} \to \mathbb{Z}_{\geq 0} 
\]
such that
\begin{itemize}
    \item $\mu_{\Gamma}(i_{j+1})=\mu_{\Gamma}(i_j)-1$ (resp. $\mu_{\Gamma}(i_{j+1})=\mu_{\Gamma}(i_j)+1$) if $e_j=(i_j,i_{j+1})$ (resp. $e_j=(i_{j+1},i_j)$) for $1 \leq j \leq l$;
    \item $\min(\{\mu_{\Gamma}(i_1),\ldots,\mu_{\Gamma}(i_l)\})=0$.
\end{itemize}
For two distinct vertices $i$ and $j$ of $G$, the distance from $i$ to $j$, denoted by ${\rm dist}_G(i,j)$, is the length of the shortest directed path in $G$ from $i$ to $j$. If there exists no directed path from $i$ to $j$, then the distance from $i$ to $j$ is defined to be infinity. The following result defines a face and its supporting hyperplane of $\mathcal{A}_G$ containing a homogenous cycle.

\begin{lemma}[{\cite[Proof of Theorem 2.2]{dimpoly}}]\label{homogcyclesupporting}
Let $G$ be a connected finite directed graph on the vertex set $[n]$ such that $\mathcal{A}_G$ is terminal and reflexive.
Assume that there exists a homogeneous cycle $\Gamma=(e_1,\ldots,e_l)$ of $G$ such that
\[\mu_{\Gamma}(i_a)-\mu_{\Gamma}(i_b) \leq {\rm dist}_G(i_a,i_b)\]
for all $1 \leq a,b \leq l$, where $e^{\rm un}_j=\{i_j,i_{j+1}\}$ for $1 \leq j \leq l$ with $i_{l+1}=i_1$.
Let $\mathcal{H} \subset \mathbb{R}^{n}$ be the hyperplane defined by the equation $a_1x_1+\cdots+a_nx_n=1$, where
\[
a_k=\begin{cases}
    \mu_{\Gamma}(k) &:\mbox{if } k \in \{i_1,\ldots,i_l\},\\
    \max(\{a_{i_j}-{\rm dist}_G(i_j,k):  j=1,\ldots, l\} \cup \{0\}) &  :\mbox{otherwise}.
\end{cases}
\]
Then $\mathcal{H}$ is a supporting hyperplane of $\mathcal{A}_G$ and the proper face $\mathcal{F}=\mathcal{A}_G \cap \mathcal{H}$ contains $\rho(e_1),\ldots,\rho(e_l)$.
\end{lemma}

For a finite acyclic directed graph $D=([n], A(D))$, we define a lattice polytope $\widetilde{\mathcal{A}}_D$ by
\[
\widetilde{\mathcal{A}}_D:={\rm conv}\{\mathbf{0}, \mathbf{e}_i-\mathbf{e}_j : (i,j) \in A(D)\} \subset \mathbb{R}^n.
\]
Let $\mathcal{A}_{D_1},\ldots, \mathcal{A}_{D_r}$ be the facets of $\mathcal{A}_G$ with acyclic directed graphs $D_1,\ldots, D_r$.
Since the origin of $\mathbb{R}^n$ belongs to the interior of $\mathcal{A}_G$, the directed edge polytope $\mathcal{A}_G$ is divided by $\widetilde{\mathcal{A}}_{D_1},\ldots, \widetilde{\mathcal{A}}_{D_r}$. In particular, each face of $\widetilde{\mathcal{A}}_{D_i}$ which does not contain the origin is a face of $\mathcal{A}_{D_i}$, hence a face of $\mathcal{A}_G$.
We recall a combinatorial description of the faces of $\widetilde{\mathcal{A}}_D$ from \cite{Setiabrata}.
\begin{definition}
Let $D=(V(D),A(D))$ be a finite acyclic directed graph and $H \subset D$ directed subgraph of $D$ with $V(H)=V(D)$ and the directed edge set $A(H)$. 
Let $H^{{\rm un}}_1,\ldots,H^{{\rm un}}_m$ be the connected components of the underlying undirected graph $H^{{\rm un}}$ of $H$. Here $H$ may have isolated vertices and we regard them as connected components of $H^{{\rm un}}$.  The directed multi-graph $H_{\rm comp}$ is the graph with vertex set
	\[
	V(H_{\rm comp})=\{H^{{\rm un}}_i : i \in [m]\}
	\]
	and edge multiset
	\[
	A(H_{\rm comp})=\{ (H^{{\rm un}}_i,H^{{\rm un}}_j) : (v_i,v_j) \in A(D) \setminus A(H), v_i \in V(H^{{\rm un}}_i), v_j \in V(H^{{\rm un}}_j) \}.
	\]
\end{definition}

\begin{figure}[h]
  \begin{subfigure}[b]{0.39\textwidth}
    \includegraphics[width=\textwidth]{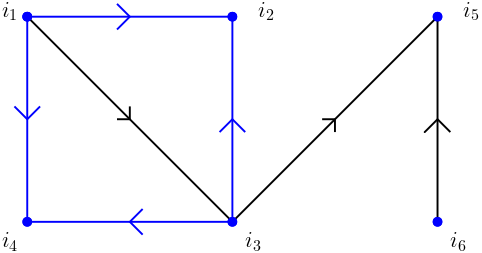}
    \caption{Directed graph $D$}
    \label{fig:f1}
  \end{subfigure}
\hskip2cm
  \begin{subfigure}[b]{0.39\textwidth}
    \includegraphics[width=\textwidth]{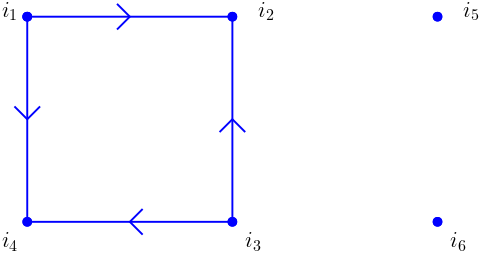}
    \caption{Subgraph $H$}
    \label{fig:f2}
  \end{subfigure}
  \caption{}
\end{figure}

\begin{example}
Let $D$ be the directed graph given in Figure \ref{fig:f1}. Consider its subgraph $H$ given in Figure \ref{fig:f2}. The directed graph $H_{\rm comp}$ has three vertices  $H^{{\rm un}}_1, H^{{\rm un}}_2$ and $H^{{\rm un}}_3$ corresponding  to connected components $H_1, H_2$ and $H_3$ of $H$. Note that $V(H_1)=\{i_1,i_2,i_3,i_4\}$, $H_2=\{i_5\}$ and $H_3=\{i_6\}$.

\begin{figure}[h]
    \centering
    \includegraphics[width=0.45\textwidth]{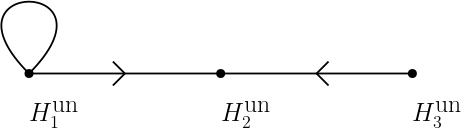}
    \caption{$H_{\rm comp}$}
    \label{fig:f3}
\end{figure}
There is a loop at vertex $H^{{\rm un}}_1$ since $(i_1,i_3) \in A(D) \setminus A(H)$ where $i_1,i_3 \in H_1$. Additionally, $H_{\rm comp}$ has two edges $(H^{{\rm un}}_1, H^{{\rm un}}_2)$ and $(H^{{\rm un}}_3, H^{{\rm un}}_2)$ because $(i_3,i_5)\in  A(D) \setminus A(H)$ where~$i_3~\in H_1$ and $i_5\in H_2$, and $(i_6,i_5)\in  A(D) \setminus A(H)$ where $i_6\in H_3$ and $i_5\in H_2$.

\end{example}

\begin{definition}
    	Let $H$ be a finite acyclic directed graph with $n$ vertices. For a pair of vertices $i,j \in [n],$ consider an undirected path $p^{\rm un}$ in $H^{\rm un}$ connecting $i$ to $j$. Let $i=i_1, i_2, \ldots,  i_k=j$ be vertices of $p^{\rm un}$ (in order) where $\{i_1,\ldots, i_k\} \subseteq [n]$. Let $p$ be a subset of $A(H)$ whose underlying undirected graph is $p^{\rm un}.$ We define the \emph{net length} of $p$ as 
    	$${\rm net}  (p) = 	|\{ (i_l,i_{l+1}) \in p : l \in [k]\}|-|\{(i_{l+1},i_l) \in p : l \in [k] \}|.$$
    	In other words, the net length of $p$ is the difference between the number of ``correctly" oriented edges and the number of ``incorrectly" oriented edges in $p$.
\end{definition}

\begin{definition}[path consistent]
	Let $H$ be a finite acyclic directed graph  with $n$ vertices. We say $H$ is \emph{path consistent} if, for any pair $i,j \in [n]$ and two undirected paths $p^{\rm un}$ and $q^{\rm un}$ in $H^{\rm un}$ connecting $i$ to $j$, we have
	\[
	{\rm net}  (p) = 	{\rm net}  (q),
	\]
	where $p$ and $q$ are the subsets of $A(H)$ whose underlying undirected graphs are the paths $p^{\rm un}$ and $q^{\rm un}$ in $H^{\rm un}$.
\end{definition}

 Note that $H$ is path consistent, if the difference between the number of correctly oriented edges and the number of incorrectly oriented edges in any path depends only on $i$ and $j$.

\begin{remark}
The homogeneous cycles are path consistent by definition.
\end{remark}

\begin{example}\label{ex:pathConsistent}
Consider the following finite directed acyclic graph $C_1$ given in Figure \ref{fig:f4}.

\begin{figure}[h]
    \centering
    \includegraphics[width=0.25\textwidth]{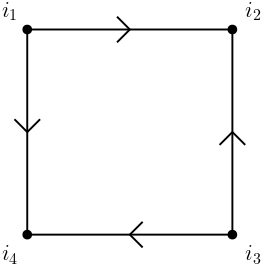}
    \caption{Path consistent directed graph $C_1$}
    \label{fig:f4}
\end{figure}

Consider the paths between the vertices $i_1$ and $i_4$. The first path has only one directed edge $(i_1,i_4)$ which has one correctly oriented edge and no incorrectly oriented edges. The second path consists of the directed edges $(i_1,i_2), (i_3,i_2), (i_3,i_4)$ such that the two edges $(i_1,i_2), (i_3,i_4)$ are correctly oriented whereas $(i_3,i_2)$ is incorrectly oriented.
\end{example}

\begin{definition}
Let $H$ be a path consistent graph such that $H^{\text{un}}$ is connected. For any two vertices $u, v \in V(H),$ pick any undirected path $p^{\text{un}}$ connecting $u$ to $v$ and  define
\[ \ell_{uv} = {\rm net } (p) \]
where $p$ is a subset of $A(H)$ whose underlying undirected graph is $p^{\rm un}$ in $H^{\rm un}.$ 

We call a vertex $u_{*} \in V(H)$ a \emph{weight source} if there is a vertex $v_{*} \in V(H)$ such that 
\[ \ell_{u_{*}v_{*}} = \text{max}_{u,v} \ell_{uv}\]
\end{definition}
\noindent Note that $\ell_{uv}$ is well-defined because $H$ is path consistent. Moreover, a weight source always exists but it not necessarily unique. As we shall see in Proposition \ref{prop:weight}, the choice of a weight source does not matter.
\begin{definition}
Let $H$ be a path consistent graph such that $H^{\text{un}}$ is connected. Let $u_{*}$ be a weight source. The weight function (with respect to $u_{*}$) of $H$ is defined as
\begin{align*}
\omega_{u_{*}} \colon V(H) & \rightarrow \mathbb{Z} \\
  i & \mapsto \ell_{u_{*},i}.
\end{align*}
\end{definition}

\begin{proposition}[{\cite[Proposition 3.13]{Setiabrata}}]\label{prop:weight}
Let $H$ be a path consistent graph and $H^{\text{un}}$ be connected. Let $\omega_{u_{*}}$ be the weight function of $H$ with respect to $u_{*}$. Then: 
\begin{enumerate}
    \item The equality $\omega_{u_{*}}(i)=\omega_{u_{*}}(j)-1$ holds for each directed edge $(i,j) \in A(H)$.
    \item $\omega_{u_{*}}(i) \geq 0$ for all $i\in V(H)$. The equality holds if and only if $i$ is a weight source.
    \item If $u'_{*}$ is another weight source then $\omega_{u_{*}} = \omega_{u'_{*}}$. Thus the weight function of $H$ is well-defined, i.e. independent of weight source.
\end{enumerate}
\end{proposition}

\begin{definition}
Let $H$ be a path consistent graph such that $H^{\rm un}_1, \ldots H^{\rm un}_k$ are the connected components of $H^{\rm un}$. Let $H_1, \ldots, H_k$ be directed subgraphs of $H$ such that 
\begin{itemize}
    \item $A(H)$ is the disjoint union of $A(H_1), \ldots, A(H_k)$, and
    \item underlying undirected graph of $H_j$ is $H^{\rm un}_j$ for each $j \in [k]$.
\end{itemize} 
The weight function of $H$ is the function $\omega \colon V(H) \to \mathbb{Z}$ obtained by gluing the weight functions $\omega_j \colon V(H_j) \to \mathbb{Z}$ of each $H_j$. 
\end{definition}

\begin{definition}
Let $D=(V(D),A(D))$ be a finite acyclic directed graph and $H \subset D$ a directed subgraph of $D$ with $V(H) = V(D)$ and the directed edge set $A(H)$. Assume that $H$ is path consistent. Each directed edge $e=(H^{\rm un}_i, H^{\rm un}_j) \in A(H_{\rm comp})$ corresponds to a unique directed edge $(v_i,v_j) \in A(D) \setminus A(H)$. We define the \textit{weight decrease} of $e$ to be the quantity
\[
{\rm wd}(e):=\omega(v_i)-\omega(v_j).
\]
\end{definition}

\begin{definition}[admissible]
Let $D=(V(D),A(D))$ be a finite acyclic directed graph and $H \subset D$ a directed subgraph of $D$ with $V(H) = V(D)$ and the directed edge set $A(H)$. Assume that $H$ is path consistent. We say that $H$ is \textit{admissible} (\textit{with respect to $D$}) if, for every directed cycle $C$ in $H_{\rm comp}$, the condition
\begin{equation}\label{eq:admissible}
  \sum_{e \in C} {\rm wd}(e) > -|C|  
\end{equation}
holds.
\end{definition}
\begin{example}
Consider the graph $D$ given in Figure \ref{fig:weight}. Let $H$ be a directed subgraph of $D$ given with the blue vertices and directed edges. Then $H^{\rm un}$ has three connected components $H^{{\rm un}}_1, H^{{\rm un}}_2$ and $H^{{\rm un}}_3$.  As in Example \ref{ex:pathConsistent} one can show that $H$ is path consistent. 
\begin{figure}[h]
    \centering
    \includegraphics[width=0.6\textwidth]{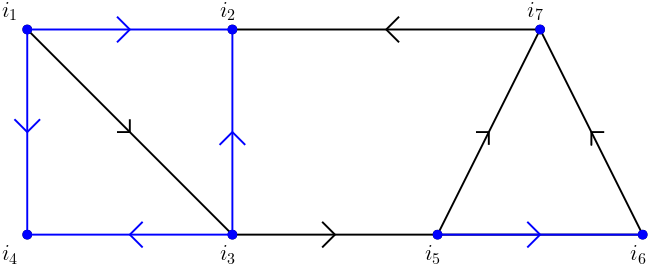}
    \caption{Directed graph $D$}
    \label{fig:weight}
\end{figure}

Let  $H_1$ be the homogeneous cycle on the vertices $i_1,i_2,i_3,i_4$, $H_2$  be the single directed edge $(i_5,i_6)$, and $H_3$ be the single vertex $i_7$ where $H^{{\rm un}}_i$ is the underlying undirected graph of $H_i$ for $i\in [3].$ As a weight source, one may choose $i_1$ or $i_3$ in $H_1$, $i_5$ in $H_2$ and $i_7$ in $H_3$. Then, by Proposition \ref{prop:weight}, values of the weight function is given as \[
\omega(i)=\begin{cases}
    1, & i \in \{i_2,i_4,i_6\},\\
    0, & \mbox{otherwise}.
\end{cases}
\]

Then the multi-graph $H_{\rm comp}$ is given as in Figure \ref{fig:weight2}. 
\begin{figure}[h]
    \centering
    \includegraphics[width=0.6\textwidth]{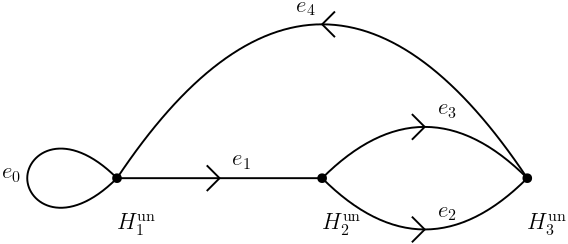}
    \caption{$H_{\rm comp}$}
    \label{fig:weight2}
\end{figure}

There are three directed cycles in $H_{\rm comp}$ where the first one is the loop $e_0$ associated to the directed edge $(i_1,i_3)$ such that ${\rm wd}(e)=\omega(i_1)-\omega(i_3)=0>-1$. Thus condition in \ref{eq:admissible} holds. The remaining directed cycles in $H_{\rm comp}$ are given by $e_1,e_2,e_4$ and $e_1,e_3,e_4$ where $e_1,e_2,e_3, e_4$ are associated to the edges $(i_3,i_5), (i_6,i_7), (i_5,i_7), (i_7,i_2)$ in $A(D)\setminus A(H)$, respectively. Let $C_1$ be directed cycle with the edges $e_1,e_2,e_4$. Then 
\begin{eqnarray*}
{\rm wd}(e_1) + {\rm wd}(e_2) + {\rm wd}(e_4)= 0 > - 3 
\end{eqnarray*}
satisfying condition \ref{eq:admissible}. Similarly, one can check that the remaining directed cycle satisfies condition \ref{eq:admissible}, as well. Thus $H$ is admissible with respect to $D$.
\end{example}

\begin{theorem}[{\cite[Theorem 3.18]{Setiabrata}}] 
\label{thm:face}
Let $D=(V(D),A(D))$ be a finite acyclic directed graph and $H \subset D$ a directed subgraph of $D$ with $V(H) = V(D)$ and the directed edge set $A(H)$.
 Then the polytope $\mathcal{A}_H$ is a face of $\widetilde{\mathcal{A}}_D$ if and only if $H$ is path consistent and admissible.
\end{theorem}

\section{Proof of Theorem \ref{thm:main}}\label{section4}
In this section, we prove Theorem \ref{thm:main}. By Proposition \ref{triangle2facerigid}, it suffices to characterize the graphs $G$ where $\mathcal{A}_G$ has only triangle 2-faces.
First, we consider the case that $G$ has the directed subgraph $C_1$.
\begin{lemma}\label{caseC1}
Let $G$ be a connected directed graph such that $\mathcal{A}_G$ is terminal and reflexive.
If $G$ has a directed subgraph $C_1$ whose directed edge set is 
\[
\{(i_1,i_2), (i_3,i_2), (i_3,i_4), (i_1,i_4)\},
\]
then $\mathcal{A}_G$ has a square $2$-face.
\end{lemma}
\begin{proof}
 The cycle $C_1=\{(i_1,i_2), (i_3,i_2), (i_3,i_4), (i_1,i_4)\}$ of $G$ is homogeneous.
 Then one has $\mu_{C_1}(i_1)=\mu_{C_1}(i_3)=1$ and $\mu_{C_1}(i_2)=\mu_{C_1}(i_4)=0$.
 Hence for all $1 \leq a,b \leq 4$, one has 
 \[\mu_{C_1}(i_a)-\mu_{C_1}(i_b) \leq {\rm dist}_G(i_a,i_b).\]
 Let $\mathcal{H} \subset \mathbb{R}^{d}$ be the hyperplane defined by the equation $a_1x_1+\cdots+a_dx_d=1$ as in Lemma \ref{homogcyclesupporting}, where
\[
a_k=\begin{cases}
    \mu_{C_1}(k) & :\mbox{if } k \in \{i_1,\ldots,i_4\},\\
    \max(\{a_{i_j}-{\rm dist}_G(i_j,k) : j=1,\ldots,4\} \cup \{0\}) & :\mbox{otherwise}.
\end{cases}
\]
Then $\mathcal{F}:=\mathcal{A}_G \cap \mathcal{H}$ is a face of $\mathcal{A}_G$ containing $\mathcal{A}_{C_1}$,
since every proper face of $\mathcal{A}_G$ is the directed edge polytope $\mathcal{A}_D$ of a finite acyclic directed graph $D$.
Let $D$ be a finite acyclic directed graph such that $\mathcal{F}=\mathcal{A}_D$.
Since $\mathcal{F}$ contains $\mathcal{A}_{C_1}$, $C_1$ is a directed subgraph of $D$. Note that by the definition of $\mathcal{H}$,  $(i_1,i_3),(i_3,i_1),(i_2,i_4),(i_4,i_2)$ are not directed edges of $D$.
Our goal is to show that $\mathcal{A}_{C_1}$ is a face of $\mathcal{A}_D$.  Let $H$ be a subgraph of $D$ on the vertex set $V(D)$ and the directed edge set $A(C_1)$ with  $|V(D)|-4$ isolated vertices. Recall from Example \ref{ex:pathConsistent} that $C_1$ is path consistent and it implies that $H$ is path consistent as well. It remains to show that $H$ is admissible with respect to $F.$

Let $H^{\rm un}_1$ be a connected component of $H^{\rm un}$ whose edge set is
\[
\{\{i_1,i_2\}, \{i_2,i_3\}, \{i_3,i_4\}, \{i_4,i_1\}\} 
\]
and $H^{\rm un}_2,\ldots,H^{\rm un}_{m-3}$ be the other connected components of $H^{\rm un}$ where $m=|V(D)|$.
Note that $H^{\rm un}_2,\ldots,H^{\rm un}_{m-3}$ are isolated vertices of $H^{\rm un}$. By choosing $i_1$ as a weight source of $C_1$ and the corresponding isolated vertices as weight sources of remaining subgraphs of $H$, we obtain the weight function of $H$ as
\[
\omega(i)=\begin{cases}
    1 & :i \in \{i_2,i_4\},\\
    0 & : \mbox{otherwise}.
\end{cases}
\]
Hence for any directed edge $e$ of $H_{\rm comp}$, one has $-1 \leq {\rm wd}(e) \leq 1$.
In particular, for a directed edge $e=(H^{\rm un}_i,H^{\rm un}_j)$ of $H_{\rm comp}$ corresponding to a directed edge $(v_i,v_j) \in A(D) \setminus A(H)$, one has ${\rm wd}(e)=-1$ if and only if $v_j \in \{i_2,i_4\}$.
Hence for any directed cycle $C$ in $H_{\rm comp}$, there is at most one directed edge $e$ in $C$ such that ${\rm wd}(e) < 0$.
This implies that for every directed cycle $C$ in $H_{\rm comp}$, one has
\[
\sum_{e \in C} {\rm wd}(e) > -|C|.
\]
Therefore, $H$ is admissible with respect to $D$.
Thus, $\mathcal{A}_H$ is a face of $\widetilde{\mathcal{A}}_D$ from Theorem \ref{thm:face}. In particular, $\mathcal{A}_H$ is a face of $\mathcal{A}_G$.
Since $\mathcal{A}_H$ is a square, $\mathcal{A}_G$ has a square $2$-face.
\end{proof}

Next, we consider the case that $G$ has a directed subgraph $C_2$.
\begin{lemma}\label{cycleC2}
Let $G$ be a connected directed graph such that $\mathcal{A}_G$ is terminal and reflexive.
Assume that $G$ has a directed subgraph $C_2$ whose directed edge set is
\[
\{(i_1,i_2), (i_2,i_3), (i_4,i_3),(i_1,i_4)\}.
\]
If $(i_1,i_3)$ is not a directed edge of $G$ and there does not exist a vertex $j$ with $j \notin \{i_2,i_4\}$ in $G$ such that $(i_1,j)$ and $(j,i_3)$ are directed edge of $G$,
then $\mathcal{A}_G$ has a square $2$-face.
\end{lemma}

\begin{proof}
  The cycle $C_2=\{(i_1,i_2), (i_2,i_3), (i_4,i_3), (i_1,i_4)\}$ of $G$ is homogeneous.
 Then one has $\mu_{C_2}(i_1)=2,\mu_{C_2}(i_3)=0$ and $\mu_{C_2}(i_2)=\mu_{C_2}(i_4)=1$.
 Since $(i_1,i_3)$ is not a directed edge of $G$, for all $1 \leq a,b \leq 4$, one has 
 \[\mu_{C_2}(i_a)-\mu_{C_2}(i_b) \leq {\rm dist}_G(i_a,i_b).\]
 Let $\mathcal{H} \subset \mathbb{R}^{d}$ be the hyperplane defined by the equation $a_1x_1+\cdots+a_dx_d=1$, where
\[
a_k=\begin{cases}
    \mu_{C_2}(k) &:\mbox{if } k \in \{i_1,\ldots,i_4\},\\
    \max(\{a_{i_j}-{\rm dist}_G(i_j,k): j=1,\ldots,4\} \cup \{0\}) & :\mbox{otherwise}.
\end{cases}
\]
Similarly to the Proof of Lemma \ref{caseC1}, we define $\mathcal{F}:=\mathcal{A}_G \cap \mathcal{H}$ a face of $\mathcal{A}_G$ containing $\mathcal{A}_{C_2}$ and we show that $\mathcal{A}_{C_2}$ is a face of $\mathcal{A}_D$.

Note that by the definition $\mathcal{H}$, $(i_1,i_3),(i_3,i_1),(i_2,i_4),(i_4,i_2)$ are not directed edges of $D$.  Then $C_2$ is path consistent. 
Now, we define the path consistent subgraph $H$ of $D$ on the vertex set $V(D)$ and the directed edge set $A(C_2)$, namely, $H$ has $|V(D)|-4$ isolated vertices. 
Let $H^{\rm un}_1$ be a connected component of $H^{\rm un}$ whose edge set is
\[
\{(i_1,i_2), (i_2,i_3), (i_4,i_3),(i_1,i_4)\} 
\]
and $H^{\rm un}_2,\ldots,H^{\rm un}_{m-3}$ be the other connected components of $H^{\rm un}$, where $m=|V(D)|$.
Note that $H^{\rm un}_2,\ldots,H^{\rm un}_{m-3}$ are isolated vertices of $H^{\rm un}$.
By choosing $i_1$ as a weight source in $C_2$ and the corresponding isolated vertices as weight sources of remaining subgraphs of $H$, we obtain the weight function of $H$ as
\[
\omega(i)=\begin{cases}
    2 & : i =i_3,\\
    1 &: i \in \{i_2,i_4\},\\
    0 & : \mbox{otherwise}.
\end{cases}
\]

Hence, for any directed edge $e$ of $H_{\rm comp}$, one has $-2 \leq {\rm wd}(e) \leq 2$. In particular, for a directed edge $e=(H^{\rm un}_i,H^{\rm un}_j)$ of $H_{\rm comp}$ corresponding to a directed edge $(v_i,v_j) \in A(D) \setminus A(H)$, one has ${\rm wd}(e)=-2$ if and only if $v_j=i_3$ and $v_i \in V(D) \setminus \{i_1,i_2,i_3,i_4\}$. Furthermore, ${\rm wd}(e)=-1$ if and only if $v_j \in \{i_2,i_4\}$ and $v_i \in V(D) \setminus \{i_1,i_2,i_3,i_4\}$.
Thus, for any directed cycle $C$ in $H_{\rm comp}$, there is at most one directed edge $e$ in $C$ such that ${\rm wd}(e) < 0$.
Suppose that there exists a directed cycle $C$ in $H_{\rm comp}$ such that
\[
\sum_{e \in C} {\rm wd}(e) \leq -|C|.
\]
Then we may assume the extreme cycle case $C=(e_1,e_2)$ with ${\rm wd}(e_1)=-2$ and ${\rm wd}(e_2)$=0.
This implies that there exists a vertex $j$ with $j \neq i_2,i_4$ in $G$ such that $(i_1,j)$ and $(j,i_3)$ are directed edge of $G$, a contradiction.
Therefore, $H$ is admissible with respect to $D$.
Thus, we conclude that $\mathcal{A}_H$ is a square 2-face of $\widetilde{\mathcal{A}}_D$ from Theorem \ref{thm:face} and hence of $\mathcal{A}_G$.
\end{proof}

Finally, we see the case where $\mathcal{A}_G$ has a non-triangle $2$-face.
\begin{lemma}
Let $G$ be a connected directed graph such that $\mathcal{A}_G$ is terminal and reflexive.
If $\mathcal{A}_G$ has a non-triangle $2$-face, then 
$G$ satisfies one of the following:
\begin{itemize}
    \item $G$ has a directed subgraph $C_1$ whose directed edge set is 
\[
\{(i_1,i_2), (i_1,i_4), (i_3,i_2), (i_3,i_4)\};
\]
\item There exists a directed subgraph $C_2$ of $G$ whose directed edge set is
\[
\{(i_1,i_2), (i_2,i_3), (i_1,i_4), (i_4,i_3)\},
\]
such that $(i_1,i_3)$ is not a directed edge of $G$ and there does not exist a vertex $j$ in $G$ such that $(i_1,j)$ and $(j,i_3)$ are directed edge of $G$.
\end{itemize}
\end{lemma}
\begin{proof}
First we classify acyclic directed graphs $D$ such that ${\rm dim}(\mathcal{A}_D)=2$.
Let $D$ be an acyclic directed graph with 
$d$ vertices without isolated vertices such that ${\rm dim}(\mathcal{A}_D)=2$ and let $D_1,\ldots,D_r$ be its connected components.
Then $d \geq 3$ and $H$ has at least $3$ edges.
Since the rank of the incidence matrix of the directed graph $D$ is $d-r$, we have $d-r-1 \leq \dim (\mathcal{A}_{D})=2$, hence, $d \leq r+3$.
Since $D$ has no isolated vertices, each connected component has at least two vertices. Hence one has $d \geq 2r$.
Therefore, we obtain $2r \leq d \leq r+3$. This implies that $r \leq 3$ and $3 \leq d \leq 6$.
If $r=1$, namely, $D$ is connected, then $d=3$ or $d=4$. Since $\dim(\mathcal{A}_D)=2$, $D$ is one of a path of length $3$, a $3$-cycle, a $4$-cycle, and a $(1,3)$-complete bipartite graph with some orientation. 
If $r=2$, then $d=4$ or $d=5$.
In this case, $D$ is a disjoint union of one edge and a path of length $2$ with some orientation. 
If $r=3$, then $d=6$. In this case, $D$ is a disjoint union of $3$ edges with some orientation.
By a routine calculation, it follows that if $D=C_1$ or $D=C_2$, then $\mathcal{A}_D$ is a square, otherwise, $\mathcal{A}_D$ is a triangle. Hence $G$ has $C_1$ or $C_2$ as subgraphs. 

Suppose that $G$ does not have $C_1$ as a subgraph.
Then we can assume that $G$ has $C_2$ as a subgraph such that $\mathcal{A}_{C_2}$ is a face of $\mathcal{A}_G$.
If $(i_1,i_3)$ is a directed edge of $G$, then $\mathcal{A}_{C_2}$ is not a face of $\mathcal{A}_G$.
In fact, the open line segment between the origin and the lattice point $\rho((i_1,i_3))$ intersects $\mathcal{A}_{C_2}$. Since the origin belongs to the relative interior of $\mathcal{A}_G$, this implies that $\mathcal{A}_{C_2}$ intersects the relative interior of $\mathcal{A}_G$. This contradicts that $\mathcal{A}_{C_2}$ is a proper face of $\mathcal{A}_G$.
Hence $(i_1,i_3)$ is not a directed edge of $G$.

In the last part of the proof, we will show that there exists no $j$ in $V(G)\setminus \{i_2,i_4\}$ such that $(i_1,j)$ and $(j,i_3)$ are directed edges of $G$. On the contrary, suppose that there exists such a vertex $j \in V(G) \setminus \{i_2,i_4\}$. Let $j_1,\ldots,j_k \in V(G) \setminus \{i_2,i_4\}$ be the vertices  of $G$ such that for any $1 \leq l \leq k$, $(i_1,j_l)$ and $(j_l,i_3)$ are directed edges of $G$.
Let $G'$ be the subgraph of $G$ whose directed edge set is
\[
\{(i_1,i_2), (i_2,i_3), (i_1,i_4), (i_4,i_3)\} \cup \{ (i_1,j_l), (j_l,i_3) : 1 \leq l \leq k \}.
\]
Observe that $G'$ is path consistent. We show that $\mathcal{A}_{G'}$ is face of $\mathcal{A}_G$.
Set $n=|V(G)|$ and $A=V(G) \setminus \{i_1,i_2,i_3,i_4,j_1,\ldots,j_k\}$.
Let $\mathcal{H} \subset \mathbb{R}^n$ be the hyperplane defined by the equation $a_1x_1+\cdots+a_n x_n=1$ and $\mathcal{H
}^{+} \subset \mathbb{R}^n$ the closed half-space defined by the inequality $a_1x_1+\cdots+a_nx_n \leq 1$, where
\[
a_l=\begin{cases}
    2 &:\mbox{if } l=i_1,\\
    1 &:\mbox{if } l \in \{i_2,i_4,j_1,\ldots,j_k\},\\
    0 &:\mbox{if } l=i_3,\\
    \max(\{a_{v}-{\rm dist}_G(v,l):  v \in V(G) \setminus A\} \cup \{0\}) &  :\mbox{if } l \in A.
\end{cases}
\]
Here, for $l \in A$,  one has $a_l=0$ if there is no $v$ with ${\rm dist}_G(v,l) < \infty$.
Then we notice that for $l \in A$, one has 
\begin{equation}
\label{eq1}
a_l \leq a_l',
\end{equation}
where $a_l'=\min(\{a_v+{\rm dist}_G(l,v):v \in V(G) \setminus A\} \cup \{0\})$. In fact, if $a_l > a_l'$, then there are $v,v' \in V(G) \setminus A$ such that
${\rm dist}_G(v,l)<\infty$, ${\rm dist}_G(l,v') < \infty$ and 
\[a_v-{\rm dist}_G(v,l) > a_{v'}+{\rm dist}_G(l,v'). \]
Since ${\rm dist}_G(v,l)+{\rm dist}_G(l,v') \geq {\rm dist}_G(v,v')$, one has that
\[
a_v-a_{v'} > {\rm dist}_G(v,l)+{\rm dist}_G(l,v') \geq {\rm dist}_G(v,v').
\]
However, since $(i_1,i_3)$ is not a directed edge of $G$, it follows that for any $v,v' \in V(G) \setminus A$, one has
\[
a_v-a_{v'} \leq {\rm dist}_G(v,v'),
\] 
a contradiction.

In this part, we observe that $\mathcal{A}_{G'}$ is a face of $\mathcal{A}_G$. First, we see that $\mathcal{H}$ is a supporting hyperplane of $\mathcal{A}_G$.
Since $\mathcal{A}_{G'} \subset \mathcal{H}$,
it is enough to show $\mathcal{A}_{G} \subset \mathcal{H}^+$. 
Let $e=(i,j)$ be a directed edge of $G$.
When $i \in V(G) \setminus A$ and $j \in A$, one has that $a_j \geq \max\{a_i-1,0\}$ by the definition of $a_j$.
Hence $a_i-a_j \leq 1$.
This implies $\rho(e) \in \mathcal{H}^+$.
For any $l \in A$, one has $a_l = 0$ or $1$.
If $i \in A$, then $a_i-a_j \leq 1$ for any $j \in V(G)$. This implies $\rho(e) \in \mathcal{H}^{+}$. Therefore, $\mathcal{A}_{G} \subset \mathcal{H}^+$. Thus $\mathcal{H}$ is a supporting hyperplane of $\mathcal{A}_G$ and the proper face $\mathcal{F}=\mathcal{A}_G \cap \mathcal{H}$ contains $\mathcal{A}_{G'}$.

Recall from Lemma \ref{lem:symedgeface} that every proper face of $\mathcal{A}_G$ is the directed edge polytope of a finite acyclic directed graph. Let $D$ be a finite acyclic directed graph such that $\mathcal{F}=\mathcal{A}_D$.
Since $\mathcal{F}$ contains $\mathcal{A}_{G'}$, $G'$ is a subgraph of $D$.
On the other hand, it is easy to see that $(i_1,i_3),(i_3,i_1),(i_2,i_4),(i_4,i_2)$ and $(i_2,j_l), (i_4,j_l), (j_l, i_2), (j_l, i_4)$ are not directed edges of $D$ for any $l=1,\ldots,k$.
We show that $\mathcal{A}_{G'}$ is a face of $\mathcal{A}_D$.
Since the origin belongs to the relative interior of $\mathcal{A}_{G}$, it is enough to show that $\mathcal{A}_{G'}$ is a face of $\widetilde{\mathcal{A}}_D$.

Let $H$ be a subgraph of $D$ on the vertex set $V(D)$ and the directed edge set $A(G')$, namely, $H$ has $|V(D)|-4-k$ isolated vertices.
Let $H^{\rm un}_1$ be a connected component of $H^{\rm un}$ whose edge set is
\[
\{\{i_1,i_2\}, \{i_2,i_3\}, \{i_3,i_4\}, \{i_4,i_1\}\}  \cup \{ \{i_1,j_l\},\{j_l,i_3\} : 1 \leq l \leq k \}
\]
and $H^{\rm un}_2,\ldots,H^{\rm un}_{m-3-k}$ be the other connected components of $H^{\rm un}$, where $m=|V(D)|$.
Note that $H^{\rm un}_2,\ldots,H^{\rm un}_{m-3-k}$ are isolated vertices of $H^{\rm un}$.
Then we have
\[
\omega(i)=\begin{cases}
    2 &  :\mbox{if }i =i_3,\\
    1 & :\mbox{if } i \in \{i_2,i_4,j_1,\ldots,j_k\},\\
    0 & :\mbox{otherwise}.
\end{cases}
\]

Hence any directed edge $e$ of $H_{\rm comp}$, one has $-2 \leq {\rm wd}(e) \leq 2$.
In particular, for a directed edge $e=(H^{\rm un}_i,H^{\rm un}_j)$ of $H_{\rm comp}$ corresponding to a directed edge $(v_i,v_j) \in A(D) \setminus A(H)$, one has ${\rm wd}(e)=-2$ if and only if $v_j=i_3$ and $v_i \in V(D) \setminus \{i_1,i_2,i_3,i_4,j_1,\ldots,j_k\}$, and ${\rm wd}(e)=-1$ if and only if $v_j \in \{i_2,i_4,j_1,\ldots,j_k\}$ and $v_i \in V(D) \setminus \{i_1,i_2,i_3,i_4,j_1,\ldots,j_k\}$.
Hence for any directed cycle $C$ in $H_{\rm comp}$, there is at most one directed edge $e$ in $C$ such that ${\rm wd}(e) < 0$.
Suppose that there exists a directed cycle $C$ in $H_{\rm comp}$ such that
\[
\sum_{e \in C} {\rm wd}(e) \leq -|C|.
\]
Then we can assume that $C=(e_1,e_2)$ and ${\rm wd}(e_1)=-2$ and ${\rm wd}(e_2)$=0.
This implies that there exists a vertex $j$ with $j \notin \{i_2,i_4,j_1,\ldots,j_k\}$ in $G$ such that $(i_1,j)$ and $(j,i_3)$ are directed edge of $G$.
However, there is no such $j$ except for $\{j_1,\ldots,j_k\}$, a contradiction.
Therefore, $H$ is admissible with respect to $D$.
Thus, $\mathcal{A}_H$ is a face of $\widetilde{\mathcal{A}}_D$ from Theorem \ref{thm:face}. In particular, $\mathcal{A}_{H}$ is a face of $\mathcal{A}_G$.
On the other hand, $\mathcal{A}_{C_2}$ is not a face of $\mathcal{A}_{G'}$.
Indeed, $C_2$ is not admissible with respect to $G'$. 
This contradicts the fact that for any two proper faces $\mathcal{F}_1$ and $\mathcal{F}_2$ of a convex polytope, $\mathcal{F}_1 \cap \mathcal{F}_2$ is a face of both $\mathcal{F}_1$ and $\mathcal{F}_2$. Because one can let $\mathcal{F}_1 = \mathcal{A}_{C_2}$ and $\mathcal{F}_2 = \mathcal{A}_{G'}$, where $\mathcal{F}_1 \cap \mathcal{F}_2=\mathcal{A}_{C_2}$.\end{proof}

This concludes the sufficient and necessary conditions where $\mathcal{A}_G$ has a non-triangle (square) 2-face and the proof of Theorem \ref{thm:main}. A simple argument proves Corollary \ref{corollaryforsymmetric} as follows.
\begin{proof}(Corollary \ref{corollaryforsymmetric})
If $X_G$ is smooth in codimension 2 and $\mathbb{Q}$-factorial in codimension 3, then the fact that $G$ has no directed subgraph $C_1$ implies that $G$ does not have a 4-cycle. On the other hand, if $G^{\rm un}$ has no $4$-cycles, then $G$ does not have $C_1$ nor $C_2$.
\end{proof}
\section{Examples and concluding remarks}\label{section5}

In this section we present two different families of directed graphs such that their associated Gorenstein toric Fano variety is not $\mathbb{Q}$-factorial (equivalently smooth by {\cite[Theorem 2.2]{dimpoly}}) but rigid. Recall that we consider connected directed graphs such that every directed edge belongs to a directed cycle. We conclude this section with some remarks about the calculation of $T^1_{X_G}$ when $\mathcal{A}_G$ a square 2-face. 
\begin{proposition}
    Let $G_{2k}$ be a symmetric directed graph with $k \geq 2$ such that $G_{2k}^{\rm un}$ is a $2k$-cycle.
    Then $\dim (\mathcal{P}_{G_{2k}})=2k-1$ and every $(2k-3)$-faces of $\mathcal{P}_{G_{2k}}$ is a simplex. 
    Hence $X_{G_{2k}}$ is $\mathbb{Q}$-factorial in codimension $2k-2$.
    In particular, $X_{G_{2k}}$ is rigid when $k \geq 3$.
\end{proposition}
\begin{proof}
This follows from \cite[Proposition 4.3]{manyasp}.
\end{proof}
This proposition presents a family of directed graphs $G$ such that $X_G$ is rigid and $\mathbb{Q}$-factorial in higher codimension. Higashitani gave a characterization of all directed graphs whose associated Gorenstein toric Fano varieties are $\mathbb{Q}$-factorial. On the other hand, Theorem~\ref{thm:main} is a characterization of all directed graphs whose associated Gorenstein toric Fano varieties are $\mathbb{Q}$-factorial in codimension $3$.
Then the following problem naturally occurs.
\begin{problem}
For any positive integer $k \geq 4$, characterize all directed graphs whose associated Gorenstein toric Fano varieties are $\mathbb{Q}$-factorial in codimension $k$.
\end{problem}
Next, we introduce another family of non-smooth rigid Gorenstein toric Fano varieties.
\begin{proposition}
Let $G$ be a connected directed graph without multiple edges constructed by gluing directed $2k$-cycle and $2l$-cycle along any number of edges. Then $X_G$ is rigid. 
\end{proposition}
\begin{proof}
If $k \geq 3$ or $l \geq 3$, it suffices to consider the case $k=2$ and $l\geq 3$. Then the two even cycles can be glued along one, two or three edges. In all cases, $G$ has no directed subgraph $C_1$ and $C_2$. If $k=l=2$, then the only case where $G$ contains $C_2$ (or $C_2$) is the case where we glue the cycles along two edges. By Theorem \ref{thm:main}, since there exists a vertex $j$ with satisfied property, $X_G$ is rigid.
\end{proof}

The investigation of the necessary condition for rigidity of $X_G$ is more challenging. If the directed edge polytope $\mathcal{A}_G$ has a square 2-face characterized as in Section \ref{section4}, one can consult the comparison theorems of Kleppe in \cite[Theorem 3.9]{kleppe} to relate $T^1_{X_P}$ to the degree zero part of $T^1_{\cone(X_P)}$. Namely one obtains the following isomorphism
$$T^1_{X_P} \cong (T^1_{\cone(X_P)})_{0}.$$
Here $\cone(X_P)$ is the cone over $X_P$ and is an affine Gorenstein toric variety. We consider the associated cone to this affine toric variety as $\cone((\mathcal{A}_G),1)) \subseteq N_{\mathbb{R}} \oplus \mathbb{R} \cong \mathbb{R}^{n+1}$. The first-order deformations of affine (Gorenstein) toric varieties has been studied by Altmann and it is known that $T^1$ admits an $M$-grading \cite[Theorem 2.3]{altmann2}. The degree zero part means that we consider the multidegrees $R\in M \times \{0\} \cong \mathbb{Z}^{n+1}$. \\

Any $\mathbb{Q}$-Gorenstein affine toric variety smooth in codimension 2 and $\mathbb{Q}$-factorial in codimension 3 is rigid \cite[Corollary 6.5.1]{altmann}. Moreover by \cite[Corollary 6.5]{altmann}, the existence of a non-triangle 2-face implies that $\dim(T^1_{\cone(X_G)})=\infty$ for $\dim(\mathcal{A}_G)\geq 3$. However, this fact does not directly guarantee the existence of a degree zero component as explained in \cite[Proposition 3.9]{portakal} for $\dim(\mathcal{A}_G)\geq 4$. In particular for $\dim(\mathcal{A}_G)=3$, since $\mathcal{A}_G$ is reflexive, one obtains a non-zero homogenous component of $T^1_{\cone(X_G)}$ for the multidegree $(m,0)$ where $m \in M \cong\mathbb{R}^n$ is defining the affine supporting hyperplane of the square 2-facet. This computation can be done by following the combinatorial recipes presented in \cite{altmann2, altmann} which we do not present in detail here. In general, this question is not trivial and needs to be explored.
\begin{example}
Let us consider the four dimensional symmetric edge polytope $\mathcal{P}_G$ of the graph $G$ as in the Figure \ref{fig:f5}. 
\begin{figure}[h]
    \centering
    \includegraphics[width=0.3\textwidth]{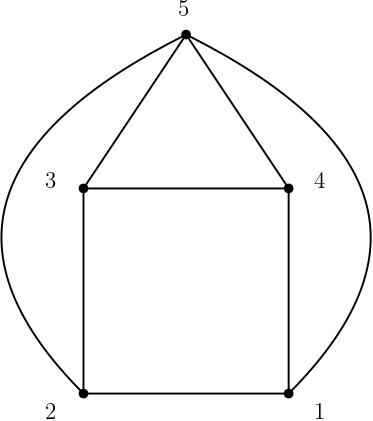}
    \caption{Graph G}
    \label{fig:f5}
\end{figure}

The symmetric edge polytope $\mathcal{P}_G$ has 10 square 2-faces, but each of them have lattice distance 2 to the origin. However for the multidegree $R=[a+1,a,a+1,a,a,0] \in M \times \{0\}$, for $a \in \mathbb{Z}$, one obtains that $T^1_{X_G}(-R) \neq 0$, hence $X_G$ is not rigid.
\end{example}

\begin{problem} Suppose that $\mathcal{A}_G$ has a square 2-face. Determine whether $X_G$ is rigid or not.
\end{problem}
In the perspective of this paper, this is a computational open question which would conclude the classification of all directed graphs $G$ such that $X_G$ is rigid.

\section*{Acknowledgements}
The authors would like to thank Matej Filip, Andrea Petracci and Linus Setiabrata for many helpful discussions. 
The third author was partially supported by JSPS KAKENHI 19J00312 and 19K14505.
\bibliographystyle{plain}
\bibliography{references}
\end{document}